\documentclass[12pt]{amsart}
\usepackage{amsmath,amsthm,amsfonts,amscd,amssymb,eucal,latexsym,mathrsfs,stmaryrd,enumitem}
\usepackage[all]{xy}

\newtheorem{theorem}{Theorem}[section]

\newtheorem{lemma}[theorem]{Lemma}
\newtheorem{proposition}[theorem]{Proposition}

\theoremstyle{definition}
\newtheorem{definition}[theorem]{Definition}

\newcommand{\Ad}{{\rm Ad}\,}

\newcommand{\id}{{\rm id}}

\newcommand{\cB}{{\mathscr B}}
\newcommand{\cE}{{\mathscr E}}
\newcommand{\cH}{{\mathscr H}}
\newcommand{\cK}{{\mathscr K}}

\newcommand{\cL}{{\mathscr L}}

\newcommand{\cU}{{\mathscr U}}

\newcommand{\cJ}{{\mathscr J}}
\newcommand{\cT}{{\mathscr T}}

\newcommand{\cR}{{\mathscr R}}

\newcommand{\Cb}{{\mathbb C}}

\newcommand{\Rb}{{\mathbb R}}
\newcommand{\Nb}{{\mathbb N}}

\newcommand{\oU}{{\mathbf{U}}}
\newcommand{\oV}{{\mathbf{V}}}

\newcommand{\WM}{{\rm WM}}

\newcommand{\eps}{\varepsilon}
\newcommand{\unit}{\mathbf{1}}

\newcommand{\overbar}[1]{\mkern 1.5mu\overline{\mkern-1.5mu#1\mkern-1.5mu}\mkern 1.5mu}

\newcommand{\oZ}{{\mathbf{Z}}}

\newcommand{\G}{G}
\newcommand{\univ}{{\rm{u}}}

\newcommand{\fin}{{\rm fin}}

\DeclareMathOperator{\Irr}{Irr}
\DeclareMathOperator{\Rep}{Rep}

\allowdisplaybreaks

\begin{document}

\title{Quantum groups, property (T), and weak mixing}

\author {Michael Brannan}
\address{Michael Brannan,
Department of Mathematics,
Mailstop 3368, Texas A\&M University, 
College Station, TX 77843-3368, USA}
\email{mbrannan@math.tamu.edu}

\author {David Kerr}
\address{David Kerr,
Department of Mathematics,
Mailstop 3368, Texas A\&M University, 
College Station, TX 77843-3368, USA}
\email{kerr@math.tamu.edu}

\begin{abstract}
For second countable discrete quantum groups, and more generally second countable
locally compact quantum groups with trivial scaling group,
we show that 
property (T) is equivalent to every weakly mixing unitary representation
not having almost invariant vectors. 
This is a generalization of a theorem of Bekka and Valette
from the group setting and was previously established
in the case of low dual by Daws, Skalsi, and Viselter.
Our approach uses spectral techniques and is completely different from 
those of Bekka--Valette and Daws--Skalski--Viselter. By a separate argument we
furthermore extend the result to second countable nonunimodular locally compact quantum groups,
which are shown in particular not to have property (T), generalizing a theorem of Fima
from the discrete setting.
We also obtain quantum group versions of 
characterizations of property (T) of Kerr and Pichot in terms of the Baire category theory
of weak mixing representations and of Connes and Weiss in term of
the prevalence of strongly ergodic actions.
\end{abstract}

\maketitle


\section{Introduction}

Introduced by Kazhdan in the 1960s for the purpose of showing 
that many lattices are finitely generated, Property (T)
has come to play a foundational 
role in the study of rigidity in Lie groups, ergodic theory, and von Neumann algebras through work
of Margulis, Zimmer, Connes, Popa, and others \cite{Zim84,BekHarVal08,Pop07}.
Over the last twenty-five years it has been extended in stages to the realm of 
quantum groups, first via Kac algebras \cite{JoiPet92}, then in the algebraic \cite{BedConTus05}
and discrete \cite{Fim10,KyeSol12} settings, and finally in the general framework of 
locally compact quantum groups as defined by Kusterman and Vaes \cite{DawFimSkaWhi16}.
In one notable recent application, Arano showed in \cite{Ara15, Ara17} that the Drinfeld double of a
$q$-deformation of compact simple lie group has property (T)
and that this implies that the duals of these $q$-deformations 
have a central version of property (T), a fact which has inspired 
progress in the theory of C$^*$-tensor categories and underpins
Popa and Vaes's construction of subfactors with property (T) standard invariant
that do not come from groups \cite{PopVae15,NesYam16}.

By definition, a locally compact group $G$ does not have property (T)
if it admits a unitary representation which does not have a nonzero invariant vector
({\it ergodicity}) but does have a net of unit vectors which is asymptotically invariant
on each group element ({\it having almost invariant vectors}).
Because ergodicity has poor permanence properties, it can be hard to leverage
this definition so as to obtain global information about the representation theory
of a group without property (T), and in particular to determine to what extent
the kind of flexible behaviour exhibited by amenable groups persists
in this more general setting. 
Bekka and Valette provided a remedy for this in the separable case
by showing that one can equivalently replace {\it ergodicity} 
above with {\it weak mixing}, which is characterized by the absence of nonzero finite-dimensional
subrepresentions, or alternatively by the ergodicity of the tensor product
of the representation with its conjugate \cite{BekVal93}. 
This leads for example to a short proof of a theorem
of Wang that characterizes property (T) in terms of the isolation of
finite-dimensional representations in the spectrum (\cite{BekVal93}, Section~4)
and a streamlined proof of the Connes--Weiss characterization of property (T)
in terms of strongly ergodic probability-measure-preserving actions (\cite{BekHarVal08}, Section~6.3).

Using the fact that weak mixing is preserved under tensor products with arbitrary 
representations, Kerr and Pichot applied the Bekka--Valette theorem
to show that if a second countable locally compact group does not have property (T) then
within the set of all unitary representations of the group on a fixed separable infinite-dimensional 
Hilbert space the weakly mixing ones 
form a dense $G_\delta$ in the weak topology \cite{KerPic08}, strengthening a
result of Glasner and Weiss that gave the 
same conclusion for ergodic representations \cite{GlaWei97}.
The idea is that any representation will approximately absorb a representation 
with almost invariant vectors under tensoring (since locally it is as if we were tensoring
with the trival representation) and so such a tensor product will be ``close'' to the
original representation while also inheriting any properties of the second one
that are preserved under tensoring, such as weak mixing.
By a similar principle requiring a more subtle implementation,
Kerr and Pichot
also established an analogous conclusion for the measure-preserving actions of the group
on a fixed standard atomless probability space.

Using the theory of positive-definite functions as in \cite{Jol05,PetPop05}, 
Daws, Skalsi, and Viselter demonstrated in \cite{DawSkaVis16}
that the conclusion of the Bekka--Valette theorem also
holds for second countable discrete unimodular quantum groups 
with low dual, and as an application they derive analogues of the Connes--Weiss theorem
and the representation-theoretic Kerr--Pichot theorem. 
Low dual is the rather restrictive assumption that there is a bound on the 
dimensions of the irreducible representations of the quantum group,
and the authors of \cite{DawSkaVis16} wonder, somewhat pessimistically, whether it can be removed. 
In the present paper we show that the Bekka--Valette and Kerr--Pichot theorems actually hold 
for all second countable discrete quantum groups, and even more generally
for all second countable locally compact quantum groups with trivial scaling group
(Theorems~\ref{T-BV} and \ref{T-KP})
as well as for all second countable nonunimodular locally compact quantum groups
(Theorem~\ref{T-BV nonunimodular}).
The methods of Daws, Skalsi, and Viselter can then also be applied 
to extend their version of the Connes--Weiss theorem to 
all second countable locally compact quantum groups with 
trivial scaling group (Theorem~\ref{T-CW}).

Our approach is completely different from those of Daws--Skalsi--Viselter
and Bekka--Valette and consists in applying the quantum group version of 
Wang's characterization of property (T) mentioned above in order to
reduce the problem to a purely spectral question concerning C$^*$-algebras. 
In Theorems~\ref{T-dense} and \ref{T-closed} we prove
that the following hold for a separable unital C$^*$-algebra $A$ and 
a fixed separable infinite-dimensional Hilbert space $\cH$:
\begin{enumerate}
\item if the spectrum of $A$ contains
no isolated finite-dimensional representations then the set of weakly mixing unital
representations of $A$ on $\cH$ is a dense $G_\delta$, and

\item if the set of finite-dimensional representations in the spectrum of $A$
is nonempty and contains only isolated points then the set of weakly mixing unital
representations of $A$ on $\cH$ is closed and nowhere dense.
\end{enumerate}
A version of the argument establishing (i) for unitary representations of countable discrete groups
has also been included in the book \cite{KerLi16} by Li and the second author. 
Theorems~\ref{T-BV} and \ref{T-KP} then follow from (i) and (ii) whenever Wang's
characterization of property (T) holds in the quantum group context, 
and this is known to be the case when the scaling group is trivial
(see Section~\ref{S-wm repr}). By a completely different argument we 
also prove in Theorem~\ref{T-BV nonunimodular} that the conclusions Theorems~\ref{T-BV} and \ref{T-KP}
are valid for second countable nonunimodular locally compact quantum groups,
which we show in particular not to have property (T), generalizing a result of Fima
from the discrete case \cite{Fim10}.

We begin in Section~\ref{S-preliminaries} by reviewing some of the
basic theory of locally compact quantum groups and their unitary representations
as developed by Kustermans and Vaes \cite{KusVae00,KusVae03,Vae01,Kus01}.
In Section~\ref{S-wm C} we study weak mixing for C$^*$-algebra representations 
and establish the two key spectral results (i) and (ii)
concerning separable unital C$^*$-algebras.
In Section~\ref{S-wm repr} we discuss weak mixing and property (T) for quantum groups,
record the quantum group incarnation of Wang's theorem, and then establish 
our versions of the Bekka--Valette and Kerr--Pichot theorems. 
Section~\ref{S-wm actions} contains the Connes--Weiss-type dynamical characterization
of property (T).
Finally, the nonunimodular case is treated in Section~\ref{S-nonunimodular}.
\medskip

\noindent{\it Acknowledgements.}
M.B. was partially supported by NSF grant DMS-1700267. D.K. was partially supported by NSF grant DMS-1500593.

\section{Preliminaries}\label{S-preliminaries}

For a C$^*$-algebra $A$ we write $M(A)$ for its multiplier algebra. 
A representation of $A$ is understood to mean a $^*$-homomorphism 
from $A$ into the C$^*$-algebra of bounded linear operators on some Hilbert space.
When working with tensor products of Hilbert spaces $\cH$ and $\cK$, 
we denote by $\Sigma$ the tensor flip map from $\cH \otimes \cK$ to $\cK \otimes \cH$.    
For linear operators on multiple tensor products, we use leg notation. 
For example, if $U$ is a unitary operator on 
a Hilbert space tensor product $\cH \otimes \cK$ we write $U_{13}$ 
for the unitary operator on a Hilbert space tensor product of the form $\cH \otimes \cJ \otimes \cK$
which is given by $V(U\otimes\id )V^{-1}$ where $V$ is the shuffle map
$\cH \otimes \cK \otimes \cJ \to \cH \otimes \cJ \otimes \cK$
defined on elementary tensors by $\xi\otimes\zeta\otimes\kappa \mapsto \xi\otimes\kappa\otimes\zeta$,
i.e., $V = \id_\cH \otimes\Sigma$.

\subsection{Locally compact quantum groups}

Our main references for generalities on locally compact quantum groups 
are \cite{KusVae00, KusVae03, Vae01}. Formally speaking, a (von Neumann algebraic) 
locally compact quantum group is a von Neumann algebra
with coassociative coproduct and left and right Haar weights,
but as usual we use the simple notation $\G$ so that we can conveniently 
and suggestively refer to the various objects that are canonically attached 
to it just as one does for locally compact groups, 
although there is no longer
anything like an underlying group. The von Neumann algebra itself
is thus written $L^\infty (\G )$, and the coproduct is a unital normal
$^*$-homomorphism $\Delta : L^\infty (\G )\to L^\infty (\G )\overbar{\otimes}L^\infty (\G )$
satisfying the coassociativity condition
\begin{align*}
(\Delta\otimes\id )\Delta = (\id\otimes\Delta )\Delta .
\end{align*}
The left and right Haar weights are normal semifinite weights $\varphi$ and $\psi$
on $L^\infty (\G )$ such that for every $\omega\in L^\infty (\G )_*^+$ one has
\begin{align*}
\varphi ((\omega\otimes\id )\Delta (a)) = \varphi (a) \omega (1)
\end{align*}
for all $a\in L^\infty (\G )^+$ with $\varphi (a) < \infty$ and
\begin{align*}
\psi ((\id\otimes\omega )\Delta (a)) = \psi (a) \omega (1)
\end{align*}
for all $a\in L^\infty (\G )^+$ with $\psi (a) < \infty$.  
The predual of $L^\infty(\G)$ is written as $L^1(\G)$, and becomes a completely contractive 
Banach algebra with respect to the {\it convolution product} 
\[
\omega_1 \star \omega_2 = (\omega_1 \otimes \omega_2) \circ \Delta , \qquad \omega_1, \omega_2 \in L^1(\G).
\]

Associated to $\G$ is a canonical weakly dense sub-C$^*$-algebra of $L^\infty (\G )$,
written $C_0 (\G )$, which plays the role of the C$^*$-algebra of continuous functions vanishing 
at infinity in the case of ordinary groups.  
We say that $\G$ is {\it second countable} if $C_0 (\G )$ is separable.  
The coproduct restricts to a unital $^*$-homomorphism $\Delta: C_0(\G) \to M(C_0(\G) \otimes C_0(\G))$. 
The algebras $C_0 (\G )$ and $L^\infty (\G )$
are standardly represented on the GNS Hilbert space $L^2 (\G )$ associated to the left Haar weight.
In the case of a locally compact group, the notations $L^\infty (\G )$, $L^1(\G)$, $C_0 (\G )$,
and $L^2 (\G )$ have their ordinary meaning.

There is a (left) {\it fundamental unitary operator} $W$ on $L^2(\G)\otimes L^2(\G)$ which
satisfies the {\it pentagonal relation} $W_{12} W_{13} W_{23} = W_{23} W_{12}$
and unitarily implements the coproduct $\Delta$ on $L^\infty(\G)$ via the formula 
$\Delta(x) = W^*(1\otimes x)W$.  Using $W$ one has 
$C_0(\G) =  \overline{\{(\id \otimes \omega)W: \omega \in \cB(L^2(\G))_*\}}^{\|\cdot\|}$, 
and one can define the {\it antipode} of $\G$ as the (generally only densely defined) 
linear operator $S$ on $C_0 (\G)$ (or $L^\infty(\G)$) satisfying the identity $(S \otimes \id )W = W^*$.   
The antipode admits a polar decomposition $S = R\circ\tau_{-i/2}$ where $R$ is an antiautomorphism 
of $L^\infty (\G )$
(the {\it unitary antipode}) and $\{ \tau_t \}_{t\in\Rb}$ is a one-parameter
group of automorphisms (the {\it scaling group}).
In the case of a locally compact group, the scaling group is trivial
and the antipode is the antiautomorphism sending a function $f \in C_0(G)$ to the function $s\mapsto f(s^{-1} )$.  
Using the antipode $S$ one can endow the convolution algebra $L^1(\G)$ with a densely defined 
involution by considering the norm-dense subalgebra $L^1_\sharp (\G)$ of $L^1(\G)$ consisting
of all  $\omega\in L^1(\G)$ for which there exists an $\omega^\sharp \in L^1(\G)$ with
$\langle \omega^\sharp, x\rangle = \overline{ \langle \omega, S(x)^* \rangle}$ for each $x\in \mathcal{D}(S)$.
It is known from  \cite{Kus01} and Section~2 of \cite{KusVae03} that $L^1_\sharp(\G)$ 
is an involutive Banach algebra with involution  $\omega\mapsto\omega^\sharp$  
and norm $\|\omega\|_\sharp = \mbox{max}\{\|\omega\|, \|\omega^\sharp\|\}$.  

Associated to any locally compact quantum group $\G$ is its 
{\it dual locally compact quantum group} $\widehat{\G}$, whose associated algebras, 
coproduct, and fundamental unitary are given by 
$C_0(\widehat \G) = 
\overline{\{(\omega \otimes \id)W: \omega \in \cB(L^2(\G))_*\}}^{\|\cdot\|} \subseteq \cB (L^2(G))$, 
$L^\infty(\widehat{\G}) = C_0(\widehat{\G})''$, $\hat \Delta (x) = \hat{W}^*(1\otimes x)\hat W$, 
and $\hat W = \Sigma W^* \Sigma$.  
Then in fact $W \in M(C_0(\G) \otimes C_0(\widehat{\G}))$,
and the {\it Pontryagin duality theorem} asserts that  
the bidual quantum group $\widehat{\widehat{\G}}$ is canonically identified with the original 
quantum group $G$.  One says that  a locally compact quantum group $\G$ 
is {\it compact} if $C_0 (\G )$ is unital, and {\it discrete} if $\widehat{\G}$ is compact, 
which is equivalent to $C_0 (\G )$ being a direct sum of matrix algebras. 

For a locally compact quantum group $\G$, we can always assume that the left and 
right Haar weights are related by $\psi = \varphi \circ R$, where $R$ is the unitary antipode.  
If the left and right Haar weights $\varphi$ and $\psi$ of $\G$ coincide then we say 
that $\G$ is {\it unimodular}. In general, the failure of $\psi$ to be left-invariant 
is measured by the {\it modular element}, which is a strictly positive element $\delta$ 
affiliated with $L^\infty(\G)$ satisfying the identities  $\Delta(\delta) = \delta \otimes \delta$ 
and $\psi(\cdot) = \varphi(\delta^{1/2}\cdot \delta^{1/2})$.    
Compact quantum groups are always unimodular, and the corresponding Haar weight can always 
be chosen to be a state. Although discrete groups are always unimodular, 
discrete quantum groups need not be. We recall that a discrete quantum group $\G$ 
is said to be of {\it Kac type} (or a {\it Kac algebra}) if it is unimodular, 
which is equivalent to the Haar state on $\widehat{G}$ being a trace.

\subsection{Unitary representations}

\begin{definition}
A {\it unitary representation} of a locally compact quantum group $\G$ on a Hilbert space $\cH$
is a unitary $U\in M(C_0 (\G ) \otimes \cK (\cH )) \subseteq 
\cB (L^2 (\G ) \otimes\cH )$ such that $(\Delta\otimes\id )(U) = U_{13} U_{23}$.
\end{definition}

In the above definition one can replace $M(C_0 (\G ) \otimes \cK (\cH ))$ 
with the larger algebra $L^\infty (\G ) \overbar{\otimes} \cB (\cH )$, for if $U$
is a unitary in the latter which satisfies $(\Delta\otimes\id )(U) = U_{13} U_{23}$ 
then $U$ automatically belongs to the former
(see for example Theorem~4.12 of \cite{BraDawSam13}).

Associated to a unitary representation $U\in M(C_0 (\G ) \otimes \cK (\cH ))$ 
is an adjointable operator on the Hilbert module $C_0 (\G )\otimes\cH$ which we write 
using the boldface version $\oU$ of the symbol in question. 
The relation between $U$ and $\oU$ is given by
\begin{align*}
\langle \oU (a\otimes\xi) , b\otimes\zeta \rangle = b^* (\id\otimes\omega_{\xi ,\zeta} )(U)a
\end{align*}
for all $a,b\in C_0 (\G )$ and $\xi ,\zeta\in\cH$, where $\omega_{\xi, \zeta}$ is
the vector functional $x\mapsto \langle x\xi , \zeta \rangle$.

Associated to $\G$ are two distinguished unitary representations,
the one-dimensional {\it trivial representation} $1_G \in M(C_0(\G))$ 
given by the unit of $L^\infty(\G)$, and the {\it left regular representation} 
given by the fundamental unitary $W \in M(C_0(\G) \otimes C_0(\widehat{\G}))$.

Two unitary representations $U\in M(C_0 (\G ) \otimes \cK (\cH_1 ))$
and $V\in M(C_0 (\G ) \otimes \cK (\cH_2 ))$ of $\G$ are {\it (unitarily) equivalent}
if there is a unitary isomorphism $u : \cH_1 \to\cH_2$ such that $V = (\id\otimes\Ad u)(U)$.  
A {\it subrepresentation} of a unitary representation $U\in M(C_0 (\G ) \otimes \cK (\cH )) $
is a unitary representation of the form 
$Q = (1\otimes P)U(1\otimes P)\in L^\infty (\G ) \overbar{\otimes} \cB (\cH_0 )$
where $\cH_0$ is a closed subspace of $\cH$, $P$ is the orthogonal projection of $\cH$ onto $\cH_0$,
and $1\otimes P$ commutes with $U$. In this case, we write $Q \le U$.  
The subrepresentation is said to be {\it finite-dimensional}
if $\cH_0$ is finite-dimensional.

Let $\G$ be a unimodular locally compact quantum group. Let 
$U\in M(C_0 (\G ) \otimes \cK (\cH )) $ be a unitary representation of $\G$.
Write $\overbar{\cH}$ for the conjugate of $\cH$, i.e., the Hilbert space which is
the same as $\cH$ as an additive group but with the scalar multiplication
$(c,\xi )\mapsto \bar{c}\xi$ for $c\in\Cb$ and inner product 
$\langle\xi,\zeta\rangle_{\overbar{\cH}} = \langle \zeta,\xi\rangle_{\cH}$.
Letting $T : \cB (\cH ) \to \cB (\overbar{\cH} )$ be the transpose map
$T(a)(\overbar{\xi} ) = \overline{a^* (\xi )}$, we define
the {\it conjugate} of $U$, written $\overbar{U}$, to be the unitary representation
\[
(R\otimes T)(U)\in M(C_0 (\G ) \otimes \cK (\overbar{\cH} ))  ) .
\]

The {\it tensor product} of two unitary representations $U\in M(C_0 (\G ) \otimes \cK (\cH )) $
and $V\in M(C_0 (\G ) \otimes \cK (\cK )) $ is the unitary representation 
\[
U \odot V :=U_{12} V_{13} \in M(C_0 (\G ) \otimes \cK (\cH \otimes \cK)) 
\subseteq  L^\infty (\G )\overbar{\otimes} \cB (\cH ) \overbar{\otimes} \cB (\cK ). 
\] 

There is a bijective correspondence between unitary representations 
$U\in M(C_0 (\G ) \otimes \cK (\cH ))$ and nondegenerate $^*$-representations 
$\pi_U:L^1_\sharp(\G) \to \cB(H)$ (\cite{Kus01}, Corollary~2.13).  
This correspondence is given by 
\[
\pi_U(\omega)  = (\omega \otimes \id) U \in  \cB(H) , \qquad \omega \in L^1_\sharp(\G).
\]
At the level of $^*$-representations of $L^1_\sharp(\G)$, 
the trivial representation $1_G$ corresponds to the $^*$-character 
$\omega \mapsto \omega(1)$, and the left regular representation is written as 
$\omega \mapsto \lambda(\omega)  = (\omega \otimes \id)W \in C_0(\widehat{\G}) \subseteq \cB(L^2(\G))$.  
As expected, we have the dual relations 
\[
C_0(\widehat \G) = \overline{\lambda(L^1_\sharp(\G))}^{\|\cdot \|} \quad\text{and}\quad 
C_0(\G) = \overline{\hat \lambda(L^1_\sharp(\widehat\G))}^{\|\cdot \|}
\] 
where $\hat \lambda$ is the left regular representation of $\widehat \G$.  

Let $C_0^{\univ} (\widehat \G )$ denote the universal enveloping C$^*$-algebra of $L^1_\sharp(\G)$.  
This is a universal version of $C_0(\widehat \G)$ which encodes the unitary representation 
theory of $\G$ (since its nondegenerate representations 
are in bijective correspondence with the nondegenerate $^*$-representations of $L^1_\sharp(\G)$
as bounded Hilbert space operators). 
In particular, the left regular representation $W$ gives rise to a surjective representation
$\lambda: C_0^{\univ} (\widehat \G ) \to  C_0 (\widehat \G ) \subseteq \cB (L^2 (G))$, 
and the trivial representation gives rise to the (dual) {\it counit} 
$\hat \eps_\univ : C_0^{\univ} (\widehat \G ) \to\Cb$.  
We will generally use the same symbols to denote $^*$-representations of $L^1_\sharp(\G)$ 
and their unique extensions to $C_0^{\univ} (\widehat\G )$.  

As was shown in \cite{Kus01}, $C^u_0(\widehat\G)$ admits a coproduct
$\hat \Delta_u:C^u_0(\widehat \G) \to M(C^u_0(\widehat\G) \otimes C^u_0(\widehat\G))$ 
which can be used to turn $(C^u_0(\widehat \G), \hat \Delta_u)$ into a 
{\it universal C$^*$-algebraic locally compact quantum group}.  
For our purposes, we only need the fact that $\hat \Delta_u$ allows one to express 
the tensor product $U \odot V$ of two unitary representations 
$U\in M(C_0 (\G ) \otimes \cK (\cH )) $ and $V\in M(C_0 (\G ) \otimes \cK (\cK )) $ 
in terms of the representation 
$(\pi_U \otimes \pi_V) \circ \sigma \hat \Delta_u: C_0^{\univ} (\widehat \G ) \to \cB(\cH \otimes \cK)$, 
where $\sigma$ denotes the tensor flip map on $C_0^{\univ} (\widehat \G ) \otimes C_0^{\univ} (\widehat \G )$.  

Finally, note that at the level of representations of  
$C_0^{\univ} (\widehat \G )$ (or of $^*$-representations of $L^1_\sharp(\G)$)
the notions of subrepresentation and unitary equivalence of unitary representations reduce 
to their standard meanings.  Indeed, given a unitary representation 
$U\in M(C_0 (\G ) \otimes \cK (\cH))$ and a projection $P$ in $\cB (\cH )$,  
the projection $1\otimes P$ commutes with $U$ if and only if $P$ commutes with 
$\pi_U (C_0^{\univ} (\widehat{\G} ))$,
in which case the representation $\pi_Q : C_0^{\univ} (\widehat{\G} ) \to \cB (P\cH )$
associated to $Q = (1\otimes P)U(1\otimes P)$ is given by $a\mapsto P\pi _U (a)P$.
Similarly, if $V\in M(C_0 (\G ) \otimes \cK (\cK ))$ is another unitary representation, 
then a unitary isomorphism $u : \cH \to\cK$ implements an equivalence 
between $U$ and $V$ if and only if it implements a unitary equivalence between $\pi_U$ and $\pi_V$  
in the sense that $\pi_V = \Ad u \circ \pi _U$.

\section{Weak mixing and representations of C$^*$-algebras}\label{S-wm C}

This section is purely C$^*$-algebraic and aims to establish two results concerning the 
prevalence of weak mixing among unital representations of a separable unital C$^*$-algebra
on a fixed Hilbert space (Theorems~\ref{T-dense} and \ref{T-closed}).

Throughout this section $A$ will denote a separable unital C$^*$-algebra.  
For a fixed Hilbert space $\cH$, the set of all unital representations of $A$ on $\cH$ 
will be written $\Rep (A,\cH)$.
We equip $\Rep (A,\cH)$ with the point-strong operator topology, which is equivalent to the point-weak
operator topology, and also to the point-$^*$-strong operator topology since
the strong and $^*$-strong operator topologies agree on the unitary group of $\cB (\cH )$
and $A$ is linearly spanned by its unitaries.

Points in the spectrum $\widehat{A}$, while formally defined as equivalence classes
of irreducible representations, will be thought of as actual representations via their
representatives, following convention. The set of finite-dimensional representations
in $\widehat{A}$ will be written $\widehat{A}_\fin$.

We begin with a discussion of weak mixing for unital representations of unital C$^*$-algebras.

\begin{definition}\label{D-wm algebra}
We say that a unital representation of $A$ on a Hilbert space is 
{\it weakly mixing} if it has no nonzero finite-dimensional subrepresentations.
\end{definition}

Recall that weak mixing for a unitary representation $\pi : G\to\cU (\cH )$ of a group
can be expressed in either of the following equivalent ways (see Theorem~2.23 in \cite{KerLi16},
and note that the countability assumption there is not needed):
\begin{enumerate}
\item for every finite set $\Omega\subseteq\cH$
and $\eps > 0$, there exists an $s\in G$ such that $|\langle \pi (s) \xi , \zeta \rangle | < \eps$
for all $\xi ,\zeta\in\Omega$,

\item $\pi$ has no nonzero finite-dimensional subrepresentations.  
\end{enumerate}
Since a unital C$^*$-algebra is linearly spanned by its unitaries, from (ii) we immediately obtain the following,
justifying the terminology of Definition~\ref{D-wm algebra}.

\begin{proposition}\label{P-unitary}
A unital representation of $A$ is weakly mixing if and only if its restriction
to the unitary group of $A$ is weakly mixing.
\end{proposition}

Next we consider a C$^*$-algebra version of Zimmer's notion \cite{Zim84} of weak containment 
for unitary representations of groups. 

\begin{definition}\label{D-prec}
Let $\pi : A\to\cB (\cH )$ and $\rho : A\to\cB (\cK )$ be unital representations. 
We write $\pi\prec\rho$ if for every finite set $\Omega\subseteq A$,
orthonormal set $\{ \xi_1 , \dots , \xi_n \} \subseteq\cH$, and $\eps > 0$ there
is an orthonormal set $\{ \zeta_1 , \dots , \zeta_n \} \subseteq\cK$ such that 
$|\langle \pi (a)\xi_i , \xi_i \rangle - \langle \rho (a)\zeta_i , \zeta_i \rangle | < \eps$
for all $a\in\Omega$ and $i=1,\dots , n$.
\end{definition}

This is the same as the usual notion of weak containment when the representation $\pi$
is irreducible, but is different in general. In fact $\pi$ is weakly contained in $\rho$
if and only if $\pi\prec\rho^{\oplus\Nb}$.

The following is a straightforward consequence of Definition~\ref{D-prec}. 
The version for unitary group representations
was noted in the remark after Proposition~H.2 in \cite{Kec10}.

\begin{lemma}\label{L-prec}
Let $\pi : A\to\cB (\cH )$ and $\rho : A\to\cB (\cK )$ be unital representations 
on separable infinite-dimensional Hilbert spaces.
Then $\pi\prec\rho$ if and only if $\pi\in\overline{\{ \kappa\in\Rep (A,\cH ) : \kappa\cong\rho \}}$.
\end{lemma}

Denote by $\WM (A,\cH ) \subseteq \Rep (A,\cH )$ the subcollection of all weakly mixing representations.

\begin{lemma}\label{L-G delta}
Let $\cH$ be a separable Hilbert space.
Then $\WM (A,\cH )$ is a $G_\delta$ in $\Rep (A,\cH )$.
\end{lemma}

\begin{proof}
Write $G$ for the unitary group of $A$.
Take an increasing sequence $\Omega_1 \subseteq \Omega_2 \subseteq \dots$ of finite subsets
of $\cH$ with dense union in $\cH$. For every $n\in\Nb$ define $\Gamma_n$ to be the
set of all $\varphi\in\Rep (A,\cH )$ such that there exists a $u\in G$
satisfying $|\langle \varphi (u)\xi , \zeta\rangle | < 1/n$ for all $\xi ,\zeta\in\Omega_n$.
Then $\Gamma_n$ is open, and so the set $\Gamma = \bigcap_{n=1}^\infty \Gamma_n$ is a $G_\delta$.
By the characterization of weak mixing for unitary group representations described before 
Proposition~\ref{P-unitary}, $\Gamma$ is precisely the set of all representations
whose restriction to $G$ is weakly mixing, which
is equal to $\WM (A,\cH )$ by Proposition~\ref{P-unitary}.
\end{proof}

\begin{lemma}\label{L-wm prec}
Suppose that each point in $\widehat{A}_\fin$ 
is isolated in $\widehat{A}$. Let $\rho : A\to\cB (\cH )$ be a representation in $\widehat{A}_\fin$.
Then there exists a weakly mixing representation $\theta$ of $A$ on a separable Hilbert space 
such that $\rho\prec\theta$.
\end{lemma}

\begin{proof}
We may assume that $\rho$ is not the limit of a sequence $\{ \pi_n \}$ 
of infinite-dimensional representations in $\widehat{A}$, for 
in that case the representation $\pi = \bigoplus_{n=1}^\infty \pi_n$ is weakly mixing and 
$\rho\prec\pi$. Since $\widehat{A}$ is second countable (\cite{Dix77}, Proposition~3.3.4),
we can then find a countable neighbourhood basis $\{ U_n \}_{n\in\Nb}$ for
$\rho$ in $\widehat{A}$ such that no $U_n$ contains an infinite-dimensional representation.

Let $n\in\Nb$. We construct a weakly mixing representation $\theta_n$ in $\widehat{A}$
as follows. First we argue that $U_n$ is uncountable. Suppose that this is not the case.
As $\widehat{A}$ is a Baire space (\cite{Dix77}, Theorem~3.4.13) and $U_n$ is open, 
$U_n$ is itself a Baire space. For every $\omega\in U_n$ the singleton $\{ \omega \}$ is closed
by finite-dimensionality (see Section~3.6 of \cite{Dix77}), and so by the Baire property 
there exists an $\omega_0 \in U_n$ such that $\{ \omega_0 \}$ is open, 
which means that $\omega_0$ is isolated, contradicting our hypothesis. Thus $U_n$ is uncountable.
We can consequently find a $d_n \in\Nb$ such that
$U_n$ contains uncountably many $d_n $-dimensional representations.

Fix a Hilbert space $\cH_n$ of dimension $d_n$.
Denote by $\Irr (A,\cH_n )$ the set of irreducible representations in $\Rep (A,\cH_n )$.
We observe the following:
\begin{enumerate}
\item $\Irr (A,\cH_n )$ is open in $\Rep (A,\cH_n )$,

\item every equivalence class in $\Irr (A,\cH_n )$ is closed in $\Irr (A,\cH_n )$,

\item for every open set $U\subseteq \Irr (A,\cH_n )$ the set of all elements in
$\Irr (A,\cH_n )$ which are equivalent to some element of $U$ is open.
\end{enumerate}
Assertion (iii) is clear.
To verify (i), let $\{ \pi_k \}$ be a convergent 
sequence in $\Rep (A,\cH_n )$ whose terms are not irreducible
and let us show that its limit $\pi$ is not irreducible.
For every $k$ choose a nonzero projection
$P_k \in \pi_k (A)'$ of rank less than $d_n$. 
In view of the finite-dimensionality of $\cH_n$,
we may assume by passing to a subsequence that the sequence $\{ P_k \}$ converges in $\cB (\cH_n )$, 
in which case its limit $P$ is a nonzero 
projection which commutes with $\pi (A)$. This means that $\pi$ is not
irreducible, yielding (i). Finally, to verify (ii) we let $\{ \pi_k \}$ be a convergent sequence 
in $\Irr (A,\cH_n )$ such that for every $k$
there exists a unitary operator $Z_k$ which conjugates $\pi_k$ to $\pi_1$. 
By the finite-dimensionality of $\cH_n$,
there is a subsequence $\{ Z_{k_j} \}_j$ that converges to a unitary operator $Z$,
which must then conjugate the limit of $\{ \pi_k \}$ to $\pi_1$, yielding (ii).

Since open subsets of Polish spaces are themselves Polish spaces (\cite{Kec95}, Theorem~3.11), 
we infer from (i) that $\Irr (A,\cH_n )$ is a Polish space. Assertions (ii) and (iii) then permit
us to apply a standard selection theorem (\cite{Kec95}, Theorem~12.16) which provides  
a Borel set $B_n \subseteq \Irr (\G ,\cH_n )$ of representatives for the relation of unitary equivalence.
Write $W_n$ for the set of all $\pi\in\Irr (A,\cH_n )$ which, as elements in $\widehat{A}$,
belong to $U_n$. This is clearly an open set in $\Irr (A,\cH_n )$, and it is uncountable
by our choice of $d_n$. Thus $B_n \cap W_n$ is an uncountable Borel set, which 
means that it is isomorphic as a Borel space to the unit interval
with Lebesgue measure and hence admits an atomless Borel probability measure of full support. 
Let $\mu_n$ be the push forward of this measure under the inclusion 
$B_n \cap W_n\hookrightarrow \Irr (A,\cH_n )$, 
and note that $\mu_n (C) = 0$ for every unitary equivalence class $C$ in $\Irr (A,\cH_n )$,

Setting $Y_n = \Irr (A,\cH_n )$, we next consider the Hilbert space $L^2 (Y_n ,\cH_n )$ of
(classes of) $\cH_n$-valued functions on $Y_n$ with inner product 
\[
\langle f,g\rangle = \int_{Y_n} \langle f(\pi ),g(\pi )\rangle \, d\mu_n (\pi ) . 
\]
Finally, we define the representation $\theta_n : A\to \cB (L^2 (Y_n,\cH_n ))$ by setting
\[
(\theta_n (a)\zeta )(\pi ) = \pi (a)\zeta (\pi ) 
\]
for $a\in A$, $\zeta\in L^2 (Y_n,\cH_n )$, and $\pi\in Y_n$.

We next verify that $\theta_n$ is weakly mixing. Suppose that this is not the case.
Then there is a finite-dimensional irreducible representation
$\pi : A\to \cB (\cH )$ and an isometric operator $Z : \cH \to L^2 (Y_n,\cH_n )$
such that $Z\pi (a) = \theta_n (a)Z$ for all $a\in A$. Then for a.e.\ $\rho\in Y_n$
and every $a\in A$ and $\eta\in\cH$ we have
\[
 (Z\pi (a) \eta )(\rho) 
= (\theta_i (a) Z \eta )(\rho )
= \rho (a) ((Z\eta )(\rho )) .
\]
so that the operator $Z_\rho : \cH \to\cH_n$ given by $Z_\rho \eta = (Z\eta )(\rho )$
satisfies $Z_\rho \pi (a) = \rho (a)Z$ for all $a\in A$. As $Z$ is isometric,
the operator $Z_\rho$ must be nonzero for all $\rho$ in a nonull subset of $Y_n$. But each such $\rho$ 
is equivalent to $\pi$ by irreducibility, contradicting the fact that the measure of
every unitary equivalence class is zero. Therefore $\theta_n$ is weakly mixing.

Now set $\theta = \bigoplus_{n=1}^\infty \theta_n$. Then $\theta$ is weakly mixing since each summand is
weakly mixing. It remains to show that $\rho\prec\theta$. Let $\Omega$
be a finite subset of $A$ and $\eps > 0$. Then we can find an $n\in\Nb$
such that for every $\pi\in U_n$ there is an isometry $V:\cH\to\cH_n$
such that $\| V\rho (a) - \pi (a)V \| < \eps /2$ for all $a\in\Omega$. 
Since bounded sets in $\cB (\cH_n )$ are precompact and representations are contractive,
we can find an open set $U\subseteq Y_n$ with $\mu (U) > 0$
such that for all $a\in\Omega$ and $\pi , \pi'\in U$ one has 
$\| \pi (a) - \pi' (a) \| < \eps /2$.
Choose a $\pi_0 \in U$ and an isometry $V:\cH\to\cH_n$ such that for all $a\in\Omega$
we have $\| V\rho (a) - \pi_0 (a)V \| < \eps /2$ and
hence $\| V\rho (a) - \pi (a)V \| < \eps$ for every $\pi\in U$.
Writing $\unit_U$ for the indicator function of $U$, we set 
$f = \mu_n (U)^{-1/2} \unit_U$, which is a
unit vector in $L^2 (Y_n ,\cH_n )$. Define an isometry $\tilde{V} : \cH\to L^2 (Y_n ,\cH_n )$ 
by $\tilde{V}\xi = f\otimes V\xi \in L^2 (Y_n ,\mu_n ) \otimes\cH_n \cong L^2 (Y_n ,\cH_n )$.
Then for all $a\in\Omega$ and norm-one vectors $\xi\in\cH$ we have
\begin{align*}
\| (\tilde{V}\rho (a) - \theta_n (a)\tilde{V})\xi \|^2
&= \| f \otimes V\rho (a)\xi - \theta_n (a) (f \otimes V\xi ) \|^2 \\
&= \frac{1}{\mu_n (U)} \int_U \| (V\rho (a) - \pi (a)V)\xi \|^2 \, d\mu_n (\pi ) \\
&\leq \frac{1}{\mu_n (U)} \int_U \eps^2 \, d\mu_n (\pi ) \\
&= \eps^2 ,
\end{align*}
so that $\| \tilde{V}\rho (a) - \theta_n (a)\tilde{V} \| < \eps$.
We conclude that $\rho\prec\theta$, as desired.
\end{proof}

We now come to the main theorems of this section.

\begin{theorem}\label{T-dense}
Suppose that no point in $\widehat{A}_\fin$ is isolated in $\widehat{A}$.
Let $\cH$ be a separable Hilbert space.
Then $\WM (A,\cH )$ is a dense $G_\delta$ in $\Rep (A,\cH )$.
\end{theorem}

\begin{proof}
By Lemma~\ref{L-G delta} it suffices to show the density of $\WM (A,\cH )$.
Let $\pi\in\Rep (A,\cH )$. By a maximality argument involving the collection
of direct sums of finite-dimensional subrepresentations of $\pi$,
we can write $\pi = \pi_0 \oplus \bigoplus_{i\in I} \pi_i$
where $\pi_0$ is weakly mixing and $\pi_i$ is finite-dimensional
for every $i\in I$. By decomposing further we may assume that $\pi_i$
is irreducible for each $i\in I$. 
By Lemma~\ref{L-wm prec}, for each $i\in I$ we can find a weak mixing representation
$\theta_i$ on a separable Hilbert space such that $\pi_i \prec\theta_i$.
Set $\rho = \pi_0 \oplus \bigoplus_{i\in I} \theta_i$. Then $\pi\prec\rho$,
and $\rho$ acts on a separable Hilbert space since $I$ is countable by the separability of $\cH$.
It follows by Lemma~\ref{L-prec} that $\pi$
belongs to the closure of the set of all $\kappa\in\Rep (G,\cH )$ such that $\kappa\cong\rho$,
and hence to the closure of $\WM (A,\cH )$, yielding the desired density.
\end{proof}

\begin{theorem}\label{T-closed}
Suppose that $\widehat{A}_\fin \neq\emptyset$
and each point in $\widehat{A}_\fin$ is isolated in $\widehat{A}$.
Let $\cH$ be a separable infinite-dimensional Hilbert space.
Then $\WM (A,\cH )$ is closed and nowhere dense in $G_\delta$ in $\Rep (A,\cH )$.
\end{theorem}

\begin{proof}
By assumption there exists a $\rho\in\widehat{A}_\fin$.
Let $\pi\in\Rep (A,\cH )$. Then clearly $\pi\prec\pi\oplus\rho$ and 
so by Lemma~\ref{L-prec} the representation $\pi$
belongs to the closure of the set of all $\kappa\in\Rep (G,\cH )$ such that $\kappa\cong\pi\oplus\rho$,
showing that the complement of $\WM (A,\cH )$ is dense in $\Rep (A,\cH )$.

Now let $\pi$ be a representation in $\Rep (A,\cH )$ which is not weakly mixing.
Then we can write $\pi = \pi_0 \oplus\pi_1$ where $\pi_1$ is finite-dimensional,
and we may assume that $\pi_1$ is irreducible. Now suppose that $\{ \rho_n \}$ is a sequence
in $\Rep (A,\cH )$ converging to $\rho$ and set $\rho = \bigoplus_{n=1}^\infty \rho_n$. 
Then $\pi\prec\rho$ and hence $\pi_1 \prec\rho$, which implies that
$\pi_1$ is a subrepresentation of $\rho$ 
since $\pi_1$ is isolated in $\widehat{A}_\fin$ (\cite{Wan75}, Theorem~1.7). 
Since $\pi_1$ is irreducible, there must
exist an $n\in\Nb$ such that $\pi_1$ is a subrepresentation of $\rho_n$.
We deduce from this that $\pi$ has a neighbourhood in $\Rep (A,\cH )$ which does not
intersect $\WM (A,\cH )$. This shows that the complement of $\WM (A,\cH )$
is open and hence completes the proof.
\end{proof}

\section{Weak mixing and Property (T) for quantum groups}\label{S-wm repr}

We now return to the context of quantum groups and discuss the notions of weak mixing and property (T).  
Let $\G$ be a locally compact quantum group.

\begin{definition}\label{D-wm quantum}
A unitary representation $U\in M(C_0 (\G ) \otimes \cK (\cH )) $ of $\G$ 
is {\it weakly mixing} if it contains no nonzero finite-dimensional subrepresentation.
\end{definition}

Since finite-dimensionality is preserved under the canonical correspondence between 
unitary representations of $\G$ and nondegenerate representations of $C_0^\univ (\widehat{\G} )$, 
by Proposition~\ref{P-unitary} we obtain the following.

\begin{proposition}
Let $U\in M(C_0 (\G ) \otimes \cK (\cH )) $ be a unitary representation, let $\pi_U$ be the  
corresponding representation of $C_0^\univ (\widehat{\G})$, and let $\pi_U^+$ be the canonical extension 
of $\pi_U$ to a unital representation of the unitization $C_0^\univ (\widehat{\G})^+$.  
Then $U$ is weakly mixing if and only 
if the restriction of $\pi_U^+$ to the unitary group of $C_0^\univ (\widehat{\G})^+$ is weakly mixing.
\end{proposition}

Let $\cH$ be a fixed Hilbert space. Write $\Rep (\G ,\cH )$ 
for the collection of all unitary representations of $G$
on $\cH$, and equip it with the point-strict topology it inherits as a subset of $M(C_0 (\G ) \otimes \cK (\cH ))$.
Let $\WM (\G,\cH ) \subseteq \Rep (\G ,\cH )$ be the subcollection of all weakly mixing representations.  
Write $\Rep (C_0^\univ (\widehat{\G} ),\cH )$
for the collection of nondegenerate representations of $C_0^\univ (\widehat{\G} )$ on $\cH$
(which is consistent with our notation in Section~\ref{S-wm C} for unital C$^*$-algebras).
In Proposition~5.1 of \cite{DawFimSkaWhi16} it is shown that the topology $\cT$ induced
on $\Rep (C_0^\univ (\widehat{\G} ),\cH )$ under the canonical bijection between 
$\Rep (\G ,\cH )$ and $\Rep (C_0^\univ (\widehat{\G} ),\cH )$ is the point-strict topology.
Since the strict topology and the $^*$-strong operator topology coincide on bounded
subsets of $\cB (\cH )$, this is the same as the point-$^*$-strong operator topology
on $\Rep (C_0^\univ (\widehat{\G} ),\cH )$. Since the $^*$-strong operator topology
and the strong operator topology agree on the unitary group of $\cB (\cH )$,
and a unital C$^*$-algebra is spanned by its unitaries,
we therefore have $\pi_n \to \pi$ in the topology $\cT$ on $\Rep (C_0^\univ (\widehat{\G} ),\cH )$ 
if and only if $\pi_n^+ \to \pi^+$ in the point-strong operator topology 
on $\Rep (C_0^\univ (\widehat{\G} )^+ ,\cH )$, where 
$\pi_n^+$ and $\pi^+$ are the canonical unital extensions of
$\pi_n$ and $\pi$ to the unitization $C_0^\univ (\widehat{\G} )^+$.
Combining these observations with Lemma \ref{L-G delta}, we obtain:

\begin{proposition}\label{P-WM G delta}
Let $\cH$ be a separable infinite-dimensional Hilbert space. 
Then $\WM (\G,\cH )$ is a $G_\delta$  in $\Rep (\G ,\cH )$.
\end{proposition}

Next we recall the definition of property (T).

\begin{definition}
Let $\G$ be a locally compact quantum group and let $U\in M(C_0 (\G )\otimes\cK (\cH ))$ 
be a unitary representation of $\G$.
A vector $\xi\in\cH$ is said to be {\it invariant} for $U$ if $U(\eta\otimes\xi ) = \eta\otimes\xi$
for all $\eta\in L^2 (\G)$.
We say that $U$ {\it has almost invariant vectors} if there is a net $\{ \xi_i \}_i$
of unit vectors in $\cH$ such that 
\[
\| U(\eta\otimes\xi_i ) - \eta\otimes\xi_i \| \to 0
\]
for all $\eta\in L^2 (\G )$, which by Proposition~3.7 of \cite{DawFimSkaWhi16} is equivalent to
\[
\| \pi_U (a)\xi_i - \hat\eps_{\rm{u}} (a)\xi_i \| \to 0
\]
for all $a\in C_0^{\univ} (\widehat{\G} )$.
\end{definition}

\begin{definition}
A locally compact quantum group $\G$ has {\it property (T)} if every unitary representation $\G$
having almost invariant vectors has a nonzero invariant vector.
\end{definition}

In order to establish the two main results of this section, Theorems~\ref{T-BV} and \ref{T-KP}, 
we require a characterization of property (T)
in terms of the isolation of finite-dimensional representations in the spectrum. 
Specifically, we will need the equivalence of (i) and (iv) in Theorem~\ref{T-fd isolated} below. 
For locally compact groups the equivalence between (i), (iii), and (iv) in Theorem~\ref{T-fd isolated}
is due to Wang \cite{Wan75}, and is known more generally for locally compact quantum groups
with trivial scaling group, although it does not seem to be explicitly stated 
in this generality in the literature 
(see Remark~5.3 of \cite{KyeSol12} and Section~3 of \cite{CheNg15}).
For quantum groups the idea is to adapt the argument of Bekka, de la Harpe, and Valette
for groups in Section~1.2 of \cite{BekHarVal08}. This requires Lemma~\ref{L-common subrep},
which generalizes a well known fact for locally compact groups.
The discrete case of Lemma~\ref{L-common subrep} appears in Section~2.5 of \cite{KyeSol12},
although the proof there works more generally, as observed in Proposition~7.2 of \cite{DasDaw16}. 
See also Section~3 of \cite{CheNg15} and the first paragraph of the proof of 
Proposition~3.5 in \cite{Vis16}. 

\begin{lemma}\label{L-common subrep}
Let $U$ and $V$ be finite-dimensional unitary representations of a 
second countable locally compact quantum group $\G$ with trivial scaling group.
Then $1_G \leq U\odot \bar{V}$ if and only if $U$ and $V$ contain a common nonzero subrepresentation.
\end{lemma}

Armed with Lemma~\ref{L-common subrep}, one can now establish the following result
by repeating mutatis mutandis the argument in Section~1.2 of \cite{BekHarVal08},
as was done in Section~3 of \cite{CheNg15} in the discrete unimodular case.

\begin{theorem}\label{T-fd isolated}
For a second countable locally compact quantum group $\G$ with trivial scaling group
the following are equivalent:
\begin{enumerate}
\item $\G$ has property (T),

\item $\hat \eps_\univ$ is isolated in the spectrum of $C_0^{\univ} (\widehat{\G} )$,

\item every finite-dimensional representation in the spectrum of $C_0^{\univ} (\widehat{\G} )$ is isolated,

\item there exists a finite-dimensional representation in the spectrum of $C_0^{\univ} (\widehat{\G} )$
which is isolated.
\end{enumerate}
\end{theorem}

\begin{theorem}\label{T-BV}
A second countable locally compact quantum group $\G$ with trivial scaling group has property (T) 
if and only if every weakly mixing unitary representation of $\G$ fails to have almost invariant vectors.
\end{theorem}

\begin{proof}
For the nontrivial direction, if $\G$ does not have property (T)
then by Theorem~\ref{T-fd isolated} and Lemma~\ref{L-wm prec} there is a 
weakly mixing unitary representation of $\G$ with almost invariant vectors.
\end{proof}

\begin{theorem}\label{T-KP}
Let $\G$ be a second countable locally compact quantum group with trivial scaling group. 
Let $\cH$ be a separable infinite-dimensional Hilbert space.
If $\G$ does not have property (T) 
then $\WM (\G ,\cH )$ is a dense $G_\delta$ in $\Rep (\G ,\cH )$, while if $\G$ has property (T) 
then $\WM (\G ,\cH )$ is closed and nowhere dense in $\Rep (\G ,\cH )$.
\end{theorem}

\begin{proof}
Apply Theorem~\ref{T-fd isolated} in conjunction with Theorems~\ref{T-dense} and \ref{T-closed}.
\end{proof}

\section{Property (T) and strongly ergodic actions}\label{S-wm actions}

We establish in Theorem~\ref{T-CW} a quantum group version of a 
result of Connes and Weiss \cite{ConWei80} for countable discrete groups.
It was verified by Daws, Skalski, and Viselter under the additional hypothesis
that the quantum group is discrete and has low dual (\cite{DawSkaVis16}, Theorem~9.3). 
In fact to obtain the conclusion we can simply apply the argument of Daws, Skalsi, and Viselter
by replacing their Theorem~7.3 with our Theorem~\ref{T-BV}, or rather a slight strengthening
of the latter in line with Remark~7.4 of \cite{DawSkaVis16}, as we explain below.

An {\it n.s.p.\ (normal-state-preserving) action} $\G\curvearrowright^\alpha (N,\sigma )$
is a normal injective unital $^*$-homomorphism 
$\alpha : N\to L^\infty (\G )\overbar{\otimes} N$,
where $N$ is a von Neumann algebra with a faithful normal state $\sigma$,
such that $(\id \otimes \alpha)\alpha = (\Delta \otimes \id)\alpha$ 
and $(\id\otimes\sigma )\alpha (x) = \sigma (x)1$ for all $x\in N$. 
We drop the symbol $\alpha$ if we don't need to refer to it explicitly.

Let $\G\curvearrowright^\alpha (N,\sigma )$ be an n.s.p.\ action. A bounded
net $\{ x_i \}\subset N$ in said to be {\it asymptotically invariant} 
if for every normal state $\omega \in L^1(\G)$ we have
$(\omega\otimes\id )(\alpha (x_i )) - x_i \to 0$ strongly, and {\it trivial}
if $x_i - \sigma (x_i )1 \to 0$ strongly. The action is said to be {\it strongly ergodic}
if every asymptotically invariant net is trivial.

Recall that $R$ denotes the unitary antipode of $\G$, which acts on $\cB (L^2 (\G ))$
as $R(x) = J_R x^* J_R$ where $J_R$ is the modular conjugation associated to the
left Haar weight on $L^\infty (\G )$. 
We say that a unitary representation $U \in M(C_0(\G) \otimes \cK(\cH))$ of $\G$ 
is {\it self-conjugate} (referred to as condition $\cR$ in \cite{DawSkaVis16})
if there exists an anti-unitary operator $J:\cH\to\cH$ such that
the anti-isomorphism $j : \cB (\cH ) \to \cB (\cH )$ given by $x\mapsto Jx^* J^*$
satisfies $(R\otimes j) (U) = U$.

Let $U$ be a unitary representation of $\G$ on a Hilbert space $\cH$.
Recall that the conjugate representation $\overline{U}$ on the conjugate Hilbert space 
$\overline{\cH}$ is defined as $(R\otimes T)(U)$ where $T:\cB (\cH )\to \cB (\overline{\cH} )$ 
is the map given by $T(b)\overline{\xi} = \overline{b^* \xi}$.
Let $V = U \odot \overline{U}$ be the tensor product of $U$ and $\overline{U}$.
Write $J$ for the anti-unitary operator on $\cH\otimes\overline{\cH}$ given on
elementary tensors by $J(\xi\otimes\overline{\zeta} ) = \zeta\otimes\overline{\xi}$.
and let $j : \cB (\cH ) \to \cB (\cH )$ be the anti-isomorphism given by $x\mapsto Jx^* J^*$.
Let $y = \sum_{i\in I} a_i\otimes b_i$ be a finite sum of elementary tensors in $L^\infty (\G )\otimes\cB (\cH )$.  Since $R(a_iR(a_j)) = a_jR(a_i)$ and $j(b_i\otimes T(b_j)) = b_j\otimes T(b_i)$ for all $i,j \in I$, it follows that the element  
\begin{align*}
x = y_{12}[(R \otimes T)y]_{13} = \sum_{i,j \in I}a_iR(a_j) \otimes b_i \otimes T(b_j) \in L^\infty(\G) \otimes \cB(\cH) \otimes \cB(\overline{\cH})
\end{align*}
satisfies
\begin{align*}
(R\otimes j)(x) = x .
\end{align*}
Since the maps $R\otimes j$ and $R\otimes T$ are $^*$-strongly continuous,
multiplication is $^*$-strongly continuous on the unit ball of 
$\cB (L^2 (\G )\otimes\cH\otimes\overline{\cH} )$, and $U$ is a $^*$-strong 
limit of operators of norm at most one in $L^\infty (\G) \otimes \cB (\cH )$ by Kaplansky density,
we conclude that $(R\otimes j)(V) = V$, so that $V$ is self-conjugate.

Now if $U$ has almost invariant vectors then so does $V$, as is easily seen,
and if $U$ is weakly mixing and $\G$ has trivial scaling group then $V$
is weakly mixing by Theorem~3.11 of \cite{Vis16}.
It thus follows from Theorem~\ref{T-BV} that if $\G$ has trivial scaling group
and does not have property (T) then it admits a weakly mixing self-conjugate 
unitary representation with almost invariant vectors.

This extra self-conjugacy condition is needed in the 
argument of Daws--Skalski--Viselter, who apply it
in Lemma~9.2 of \cite{DawSkaVis16} so as to permit
the use of Vaes's construction of actions on the free Araki--Wood factors from \cite{Vae05}.
Now that we have also have it in our more general setting, 
we can apply the argument in Section~9 of \cite{DawSkaVis16} to deduce the following theorem.
Note that the assumption of trivial scaling group is not merely required for the application
of Theorem~3.11 of \cite{Vis16} in the previous paragraph, but is also 
a hypothesis in Lemma~9.2 of \cite{DawSkaVis16}. Here $(\cL F_\infty , \tau )$
is the von Neumann algebra of the free group on a countably infinite set of generators
along with its canonical tracial state.

\begin{theorem}\label{T-CW}
For a second countable locally compact quantum group $\G$ with trivial scaling group
the following are equivalent:
\begin{enumerate}
\item $\G$ has property (T),

\item every weakly mixing n.s.p.\ action $\G\curvearrowright (N,\sigma )$
is strongly ergodic,

\item every ergodic n.s.p.\ action $\G\curvearrowright (N,\sigma )$
is strongly ergodic.
\end{enumerate}
One can also replace $(N,\sigma )$ with the fixed pair $(\cL F_\infty , \tau )$ in (ii) and (iii).
\end{theorem}

\section{The general nonunimodular case}\label{S-nonunimodular}

In this final section, our aim is to prove Theorem~\ref{T-BV nonunimodular}, which 
extends Theorems~\ref{T-BV} and \ref{T-KP} so as to cover the general nonunimodular case.
It shows in particular that a nonunimodular second countable locally compact quantum group
cannot have property (T), which in the discrete situation was established in \cite{Fim10}.

Let $\G$ be a second countable locally compact quantum group.  
Recall that the modular element of $\G$ is strictly positive unbounded operator 
$\delta$ on $L^2 (\G )$ affiliated with the von Neumann algebra $L^\infty (\G )$. 
The  map $t\mapsto \delta^{it}$ 
from $\Rb$ to $L^\infty(\G) \subseteq \cB (L^2 (\G ))$ is continuous in the strong operator topology 
by Stone's theorem and so it defines a unitary operator $U$ on 
$L^2 (\Rb , L^2 (\G ))\cong L^2 (\Rb )\otimes L^2 (\G )$
which, when viewing $L^2 (\Rb , L^2 (\G ))$ as the Hilbert space direct integral 
of copies of $L^2 (\G )$ over $\Rb$ with respect to Lebesgue measure, is a decomposable
operator and as such is expressed by the direct integral $\int^\oplus_\Rb \delta^{it} \, dt$.
Thus for all $\eta_1 ,\eta_2 \in L^2 (\G )$ and $\xi_1 ,\xi_2 \in L^2 (\Rb )$ we have
\begin{align}\label{E-inner product}
\langle U(\eta_1 \otimes\xi_1 ), \eta_2 \otimes\xi_2 \rangle
&= \int_\Rb \langle \xi_1 (t) \delta^{it} \eta_1 , \xi_2 (t) \eta_2 \rangle \, dt \\
&= \int_\Rb \xi_1 (t)\overline{\xi_2 (t)} \langle \delta^{it} \eta_1 , \eta_2 \rangle \, dt \notag .
\end{align}
Since $\Delta (\delta^{it} ) = \delta^{it} \otimes\delta^{it}$ 
(see for example the proof of Proposition 1.9.11 in \cite{Vae01}),
we see that $(\Delta\otimes\id )U$ and $U_{13} U_{12}$ are both decomposable operators on 
$L^2 (\Rb , L^2 (\G ) \otimes L^2 (\G )) \cong L^2 (\G ) \otimes L^2 (\G )\otimes L^2 (\Rb )$
which can be expressed as the direct integral $\int^\oplus_\Rb \delta^{it} \otimes\delta^{it} \, dt$.
Thus $U$ is a unitary representation of $\G$ on $L^2 (\Rb )$.  
(In fact, $U \in L^\infty(\G) \overline{\otimes} L^\infty(\mathbb R)$ 
can be regarded as a unitary representation of both $\G$ and $\mathbb R$ simultaneously.)

\begin{lemma}\label{L-nonunimodular}
The unitary representation $U$ of $\G$ has almost invariant vectors.
\end{lemma}

\begin{proof}
For every $n\in\Nb$, writing $\unit_{[0,\frac1n ]}$ for the indicator function
of $[0,\frac1n ]$ we set $\xi_n = \sqrt{n} \unit_{[0,\frac1n ]}$, 
which is a unit vector in $L^2 (\Rb )$.
Let $\eta\in L^2 (\G )$ be a unit vector.
Using the formula (\ref{E-inner product}), for every $n$ we have
\begin{align*}
|\langle U(\eta \otimes\xi_n ), \eta \otimes\xi_n \rangle - 1 |
= n \int_{[0,\frac1n ]} \langle (\delta^{it} - 1)\eta , \eta \rangle \, dt 
\end{align*}
and the expression on the right converges to zero as $n\to\infty$ by the strong operator continuity
of the map $t\mapsto \delta^{it}$. Thus $U$ has almost invariant vectors.
\end{proof}

Let $V$ be any unitary representation of $\G$ on a separable Hilbert space $\cH$,
and consider the tensor product representation $Z = V \odot U = V_{12}U_{13}$.
Viewing $L^2 (\Rb , L^2 (\G )\otimes\cH) \cong L^2 (\G ) \otimes \cH\otimes L^2 (\Rb )$ 
as the Hilbert space direct integral of copies of $L^2 (\G )\otimes\cH$ over $\Rb$ 
with respect to Lebesgue measure, the operator $V_{12}$ is decomposable and can
expressed by the direct integral $\int^\oplus_\Rb V \, dt$,
so that for $\eta\in L^2 (\G )$ and $\zeta\in L^2 (\Rb ,\cH ) \cong \cH\otimes L^2 (\Rb )$
the vector $V_{12}^* (\eta\otimes\zeta )$, viewed as an element of $L^2 (\Rb ,L^2 (\G ) \otimes \cH )$,
is equal to $t\mapsto V^* (\eta\otimes\zeta (t))$.
Thus for $\eta_1 , \eta_2 \in L^2 (\G )$ and 
$\zeta_1 , \zeta_2 \in L^2 (\Rb ,\cH ) \cong \cH\otimes L^2 (\Rb )$ we have the formula
\begin{align}\label{E-inner product tensor}
\langle Z(\eta_1 \otimes\zeta_1 ) , \eta_2 \otimes\zeta_2 \rangle
&= \langle U_{13}(\eta_1 \otimes\zeta_1 \rangle ) , V_{12}^* (\eta_2 \otimes\zeta_2 \rangle ) \\
&= \int_\Rb \langle \delta^{it} \eta_1 \otimes \zeta_1 (t) , 
V^* (\eta_2 \otimes \zeta_2 (t) )\rangle \, dt . \notag
\end{align}

\begin{lemma}\label{L-nonunimodular tensor}
Suppose that $\G$ is not unimodular.
Let $V$ be any unitary representation of $\G$ on a separable Hilbert space $\cH$.
Then the tensor product representation $Z = V \odot U$ is weakly mixing. 
Moreover, if $\cH$ is infinite-dimensional
then the closure of the set of unitary conjugates of $Z$ in $\Rep (\G ,\cH )$ contains $V$.
\end{lemma}

\begin{proof}
Since $\G$ is not unimodular, there is a real number 
$c\neq 0$ such that $e^c$ belongs to the spectrum of the modular element $\delta$.
The equation $\Delta (\delta ) = \delta\otimes\delta$ then implies, via elementary spectral theory, 
that $e^{nc}$ belongs to the spectrum of $\delta$ for every $n\in\Nb$.

Let $\cE$ be a nonzero finite-dimensional subspace of $L^2 (\Rb )\otimes\cH$.
Choose an orthonormal basis $\{ \zeta_1 , \dots , \zeta_K \}$ for $\cE$.
Let $\eps > 0$. Then there are a $b>0$ and an $N\in\Nb$ such that
for each $k=1,\dots , K$ we can find a $\zeta_k' \in L^2 ([-b,b] , \cH )\cong L^2 ([-b,b])\otimes\cH$
which is an $\cH$-valued step function on $[-b,b]$ taking at most $N$ different values  
and satisfying $\| \zeta_k' - \zeta_k \| < \eps /3$.
Set $M = \max_{k=1,\dots ,K} \max_{t\in [-b,b]} \| \zeta_k' (t) \|$.
By the proof of the Riemann--Lebesgue lemma for step functions, there is 
an $n\in\Nb$ depending only on $N$ and $M$ such that, setting $s = nc$, every
step function $f : [-b,b] \to\Cb$ which takes at most $N^2$ values
and is bounded in modulus by $M^2$ satisfies
\begin{align}\label{E-RL}
\bigg| \int_{[-b,b]} f(t) e^{ist} \, dt \bigg| < \frac{\eps}{6} 
\end{align}
Write $\cK$ for the range of the spectral projection of $\delta$ corresponding to $[0,2e^s]$.
Since $e^s$ belongs to the spectrum of $\delta$, we can find a norm-one vector  
$\eta\in\cK$ such that $\| \delta\eta - e^s \eta \|$ is small enough so that 
by the continuous functional calculus, applied to $\delta$ acting (boundedly) on $\cK$,
we have $\| \delta^{it} \eta - e^{ist} \eta \| < \eps /(6M^2)$ for all $t\in [-b,b]$. 
For every norm-one vector $\theta\in L^2 (\G )$ and
$k=1,\dots ,n$ the function $t\mapsto \langle \eta \otimes\zeta_1' (t) , V^* (\theta\otimes \zeta_k' (t))\rangle$
on $[-b,b]$ is a step function which takes at most $N^2$ values and is bounded in modulus by $M^2$,
and so using the formula (\ref{E-inner product tensor}) and applying (\ref{E-RL}) we obtain
\begin{align*}
|\langle Z(\eta \otimes\zeta_1' ), \theta \otimes \zeta_k' \rangle |
&= \bigg| \int_{[-b,b]} \langle \delta^{it} \eta \otimes\zeta_1' (t) , V^* (\theta\otimes \zeta_k' (t)) \rangle \, dt\bigg| \\
&\leq \bigg| \int_{[-b,b]} \langle e^{ist} \eta \otimes\zeta_1' (t) , V^* (\theta\otimes \zeta_k' (t)) \rangle \, dt\bigg| \\
&\hspace*{12mm} \ + \bigg| \int_{[-b,b]} 
\langle \delta^{it} \eta - e^{ist} \eta \otimes\zeta_1' (t) , V^* (\theta\otimes \zeta_k' (t)) \rangle \, dt\bigg| \\
&\leq \bigg| \int_{[-b,b]} \langle \eta \otimes\zeta_1' (t) , V^* (\theta\otimes \zeta_k' (t)) \rangle e^{ist} \, dt\bigg| \\
&\hspace*{12mm} \ + 
\sup_{t\in [-b,b]} \big( \| \delta^{it} \eta - e^{ist} \eta \| \| \zeta_1' (t) \| \| \zeta_k' (t)\| \big) \\
&< \frac{\eps}{6} + \frac{\eps}{6M^2} \cdot M^2 \\
&= \frac{\eps}{3} 
\end{align*}
and hence
\begin{align}\label{E-theta}
|\langle Z(\eta \otimes\zeta_1 ), \theta \otimes \zeta_k \rangle |
&\leq |\langle Z(\eta \otimes\zeta_1' ), \theta \otimes \zeta_k' \rangle | 
+ \| \zeta_1 - \zeta_1' \| + \| \zeta_k - \zeta_k' \| \\
&< \frac{\eps}{3} + \frac{\eps}{3} + \frac{\eps}{3} = \eps . \notag
\end{align}

Now given a norm-one vector $\kappa\in L^2 (\G ) \otimes\cE$ we can
write it as $\sum_{k=1}^K \theta_k \otimes \zeta_k$ where $\sum_{k=1}^K \| \theta_k \|^2 = 1$,
and so if we take $\eps = 1/K$ then from (\ref{E-theta}) we get
\begin{align*}
|\langle Z(\eta \otimes\zeta_1 ), \kappa \rangle |
\leq \sum_{k=1}^K |\langle Z(\eta \otimes\zeta_1 ), \theta_k \otimes \zeta_k \rangle | 
< \sum_{k=1}^K \frac1n \| \theta_k \| 
\leq K\eps 
= 1 .
\end{align*}
Since the vector $Z(\eta \otimes\zeta_1 )$ has norm one by the unitarity of $Z$, 
it follows that $Z(\eta \otimes\zeta_1 )\notin \cE$,
which shows that $L^2 (\G )\otimes \cE$ is not $U$-invariant.
We conclude that $U$ has no nonzero finite-dimensional subrepresentations.

Suppose now that $\cH$ is infinite-dimensional and let us show that
the closure of the set of unitary conjugates of $Z$ in $\Rep (\G ,\cH )$ contains $V$.
Fix an orthonormal basis $\{ \zeta_k \}_{k=1}^\infty$ of $\cH$.
Let $\Omega$ be a finite set of norm-one elements in $C_0 (\G )$.
Let $K\in\Nb$ and $\eps > 0$.
Working in the Hilbert module $C_0 (\G ) \otimes\cH$,
for each $k=1,\dots ,K$ there exist
$x_{k,1} , \dots , x_{k,L_k} , y_{k,1} , \dots , y_{k,L_k} \in\cH$
such that $\| \oV (a\otimes\zeta_k ) - \sum_{l=1}^{L_k} x_{k,l} \otimes \zeta_l \| < \eps /3$
and $\| \oV^* (a\otimes\zeta_k ) - \sum_{l=1}^{L_k} y_{k,l} \otimes \zeta_l \| < \eps /3$
for all $a\in\Omega$.
Set $L = \max \{ k,L_1 , \dots , L_K \}$.
Using the characterization of having almost invariant vectors 
given in Proposition~3.7(v) of \cite{DawFimSkaWhi16}, 
Lemma~\ref{L-nonunimodular} yields a unit vector $\xi\in\cH$ such that
$\| \oU (a\otimes \xi ) - a\otimes\xi \| < \eps /3$ for all $a\in\Omega$.
Since $\cH$ is separable and infinite-dimensional we can find a unitary
isomorphism $u : \cH \otimes L^2 (\Rb ) \to\cH$ which sends $\zeta_k \otimes\xi$
to $\zeta_k$ for every $k=1,\dots , L$.
Then for every $a\in\Omega$ and $k=1,\dots , K$ we have
\begin{align*}
\lefteqn{\| ((\id\otimes u)\oZ (\id\otimes u)^{-1} - \oV )(a\otimes\zeta_k ) \|}\hspace*{20mm} \\
\hspace*{20mm} &\leq \| (\id\otimes u)\oV_{12} (\oU_{13} (a \otimes\zeta_k \otimes\xi ) - a\otimes\zeta_k \otimes\xi ) \| \\
&\hspace*{10mm} \ + \bigg\| (\id\otimes u)\bigg(\oV_{12} (a \otimes\zeta_k \otimes\xi )
- \sum_{l=1}^{L_k} x_{k,l} \otimes \zeta_l \otimes \xi \bigg) \bigg\| \\
&\hspace*{10mm} \ + \bigg\| \sum_{l=1}^{L_k} x_{k,l} \otimes \zeta_l - \oV (a\otimes\zeta_k ) \bigg\| \\
&< \frac{\eps}{3} + \frac{\eps}{3} + \frac{\eps}{3} = \eps 
\end{align*}
and similarly $\| ((\id\otimes u)\oZ^* (\id\otimes u)^{-1} - \oV^* )(a\otimes\zeta_k ) \| < \eps$.
We conclude that $V$ belongs to the closure of the set of unitary conjugates of $Z$ in $\Rep (\G ,\cH )$.
\end{proof}

We can now conclude with the main result of the section.

\begin{theorem}\label{T-BV nonunimodular}
Let $\G$ be a second countable nonunimodular locally compact quantum group. 
Then there exists a a weakly mixing unitary representation of $\G$ which has almost invariant vectors.
Moreover, if $\cH$ is a separable infinite-dimensional Hilbert space 
then $\WM (\G ,\cH )$ is a dense $G_\delta$ in $\Rep (\G ,\cH )$.
\end{theorem}

\begin{proof}
The representation $U$ has almost invariant vectors by Lemma~\ref{L-nonunimodular},
and it is weakly mixing by Lemma~\ref{L-nonunimodular tensor}, as we can take $V$ there
to be the trivial representation.

Finally, if $\cH$ is a separable infinite-dimensional Hilbert space then 
$\WM (\G ,\cH )$ is a dense $G_\delta$ in $\Rep (\G ,\cH )$
by Lemmas~\ref{L-G delta} and \ref{L-nonunimodular tensor}. 
\end{proof}


\begin{thebibliography}{999}




\bibitem{Ara15}
Y. Arano. Unitary spherical representations of Drinfeld doubles.
To appear in {\it J. Reine Angew.\ Math.}

\bibitem{Ara17}
Y. Arano. Comparison of unitary duals of Drinfeld doubles and complex semisimple Lie groups. 
{\it Comm. Math. Phys.} {\bf 351} (2017), 1137--1147. 

\bibitem{BedConTus05}
E. B\'{e}dos, R. Conti, and L. Tuset.
On amenability and co-amenability of algebraic quantum groups and their corepresentations. 
{\it Canad.\ J. Math.}\ {\bf 57} (2005), 17--60. 

\bibitem{BekVal93} 
M. E. B. Bekka and A. Valette. Kazhdan's property (T) and amenable
representations. {\it Math.\ Z.} {\bf 212} (1993), 293--299.

\bibitem{BekHarVal08}
B. Bekka, P. de la Harpe, and A. Valette. {\it Kazhdan's Property (T).}
New Mathematical Monographs, 11. Cambridge University Press, Cambridge, 2008.

\bibitem{BraDawSam13}
M. Brannan, M. Daws, and E. Samei.
Completely bounded representations of convolution algebras of locally compact quantum groups.
{\it M\"{u}nster J. Math.}\ {\bf 6} (2013), 445--482. 

\bibitem{CheNg15}
X. Chen and C.-K. Ng.
Property T for locally compact groups. {\it Intl.\ J. Math.} {\bf 26} (2015), 1550024, 13 pp.

\bibitem{ConWei80}
A. Connes and B. Weiss.
Property T and asymptotically invariant sequences.
{\it Israel J. Math.}\ {\bf 37} (1980), 209--210. 

\bibitem{DasDaw16}
B. Das and M. Daws.
Quantum Eberlein compactifications and invariant means. 
{\it Indiana Univ.\ Math.\ J.} {\bf 65} (2016), 307--352. 

\bibitem{DawFimSkaWhi16}
M. Daws, P. Fima, A. Skalski, S. White.
The Haagerup property for locally compact quantum groups. 
{\it J. Reine Angew.\ Math.} {\bf 711} (2016), 189--229. 

\bibitem{DawSkaVis16}
M. Daws, A. Skalsi, and A. Viselter. Around property (T) for quantum groups. arXiv:1605.02800v1.

\bibitem{Dix77}
J. Dixmier. {\it C$^*$-Algebras.} 
Translated from the French by Francis Jellett. 
North-Holland Mathematical Library, Vol.\ 15. 
North-Holland Publishing Co., Amsterdam-New York-Oxford, 1977.

\bibitem{Fim10}
P. Fima.
Kazhdan's property T for discrete quantum groups. 
{\it Internat.\ J. Math.} {\bf 21} (2010), 47--65. 

\bibitem{GlaWei97}
E. Glasner and B. Weiss. 
Kazhdan's property T and the geometry of the collection of invariant measures. 
{\it Geom.\ Funct.\ Anal.} {\bf 7} (1997), 917--935.

\bibitem{JoiPet92}
M. Joi\c{t}a and S. Petrescu.
Property (T) for Kac algebras.
{\it Rev.\ Roumaine Math.\ Pures Appl.}\ {\bf 37} (1992), 163--178. 

\bibitem{Jol05}
P. Jolissaint.
On property (T) for pairs of topological groups.
{\it Enseign.\ Math.\ (2)} {\bf 51} (2005), 31--45. 


\bibitem{Kec95}
A. S. Kechris. {\it Classical Descriptive Set Theory.} 
Graduate Texts in Mathematics, 156. Springer-Verlag, New York, 1995.

\bibitem{Kec10}
A. S. Kechris.
{\it Global Aspects of Ergodic Group Actions.}
Mathematical Surveys and Monographs, 160. American Mathematical Society, Providence, RI, 2010. 

\bibitem{KerLi16}
D. Kerr and H. Li. {\it Ergodic Theory: Independence and Dichotomies}.
Springer, 2016.

\bibitem{KerPic08}
D. Kerr and M. Pichot.
Asymptotic Abelianness, weak mixing, and property T. 
{\it J. Reine Angew.\ Math.} {\bf 623} (2008), 213--235. 

\bibitem{Kus01}
J. Kustermans.
Locally compact quantum groups in the universal setting. 
{\it Internat.\ J. Math.}\ {\bf 12} (2001), 289--338. 

\bibitem{KusVae00}
J. Kustermans and S. Vaes.
Locally compact quantum groups. 
{\it Ann.\ Sci.\ {\'E}cole Norm.\ Sup.\ (4)} {\bf 33} (2000), 837--934. 

\bibitem{KusVae03}
J. Kustermans and S. Vaes.
Locally compact quantum groups in the von Neumann algebraic setting.
{\it Math.\ Scand.}\ {\bf 92} (2003), 68--92. 

\bibitem{KyeSol12}
D. Kyed and P. Soltan. 
Property (T) and exotic quantum group norms. 
{\it J. Noncommut.\ Geom.} {\bf 6} (2012), 773--800. 

\bibitem{Lan95}
E. C. Lance.
{\it Hilbert C$^*$-Modules. A Toolkit for Operator Algebraists}. 
London Mathematical Society Lecture Note Series, 210. 
Cambridge University Press, Cambridge, 1995.

\bibitem{NesYam16}
S. Neshveyev and M. Yamashita.
Drinfeld center and representation theory for monoidal categories. 
{\it Comm.\ Math.\ Phys.}\ {\bf 345} (2016), 385--434. 

\bibitem{PetPop05}
J. Peterson and S. Popa.
On the notion of relative property (T) for inclusions of von Neumann algebras. 
{\it J. Funct.\ Anal.} {\bf 219} (2005), 469--483. 

\bibitem{Pop07}
S. Popa. Deformation and rigidity for group actions and von Neumann algebras. 
In: {\it International Congress of Mathematicians. Vol. I.}, 445--477. 
Eur.\ Math.\ Soc., Z\"{u}rich, 2007. 

\bibitem{PopVae15}
S. Popa, and S. Vaes.
Representation theory for subfactors, $\lambda$-lattices and C$^*$-tensor categories. 
{\it Comm.\ Math.\ Phys.}\ {\bf 340} (2015), 1239--1280. 

\bibitem{Vae01}
S. Vaes. Locally compact quantum groups. Thesis, Catholic University of Leuven, 2001.

\bibitem{Vae05}
S. Vaes.
Strictly outer actions of groups and quantum groups. 
{\it J. Reine Angew.\ Math.}\ {\bf 578} (2005), 147--184. 

\bibitem{Vis16}
A. Viselter. Weak mixing for locally compact quantum groups. 
To appear in {\it Ergodic Theory Dynam.\ Systems.}

\bibitem{Wan75}
P. S. Wang.
On isolated points in the dual spaces of locally compact groups. 
{\it Math.\ Ann.} {\bf 218} (1975), 19--34. 

\bibitem{Zim84}
R. J. Zimmer.
{\it Ergodic Theory and Semisimple Groups.}
Monographs in Mathematics, 81. Birkh\"{a}user Verlag, Basel, 1984. 

\end{thebibliography}
\end{document}